\documentclass[12pt,english]{amsart}
\usepackage{amsfonts, amssymb, amsmath, amsthm, eucal, latexsym}

\usepackage[english]{babel}

\usepackage{epsfig}
\usepackage{amsmath}
\usepackage{color}
\usepackage{graphicx}
\def\div{\hbox{div}}

\newtheorem{thm}{Theorem}[section]

\newtheorem{theorem}[thm]{Theorem}
\newtheorem{definition}[thm]{Definition}
\newtheorem{lemma}[thm]{Lemma}
\newtheorem{proposition}[thm]{Proposition}

\newtheorem{rem}[thm]{Remark}

\numberwithin{equation}{section}

\makeatletter

\newcommand{\Rmnum}[1]{\expandafter\@slowromancap\romannumeral #1@}
\makeatother

\title{\bf On 2D NLS on non-trapping exterior domains}
\author{ Farah ABOU SHAKRA }
\thanks{The author is a Lebanese CNRS scholar.}
\address{Universit\'{e} Paris 13, Sorbonne Paris Cit\'{e}, LAGA, CNRS, UMR 7539, F-93430, Villetaneuse, France.
}

\email{shakra@math.univ-paris13.fr}

\begin{document}

\begin{abstract}
Global existence and scattering for the nonlinear defocusing Schr\"{o}dinger equation in 2 dimensions are known for domains exterior to 
star-shaped obstacles 
and for nonlinearities that grow at least as the quintic power. In this paper, we extend the global existence result for all non-trapping 
obstacles and for nonlinearities with power strictly greater than quartic. 
For such nonlinearities, we also prove scattering for a class of so-called almost star-shaped obstacles.
\end{abstract}

\maketitle

\section{Introduction and Background}
 We are interested in this paper in the nonlinear Schr\"{o}dinger equation in exterior domains $\Omega=\mathbb{R}^n\setminus V$ 
where $V$ is a non-trapping obstacle with smooth boundary with Dirichlet boundary condition
\begin{align}\label{schrodinger}
&i\partial_{t}u+\Delta u=\pm|u|^{p-1}u\;\;in\;\;\Omega=\mathbb{R}^{n}\setminus V,\;\;p\geq1\notag\\
&u|_{\mathbb{R}\times\partial\Omega}=0\\
&u(0,x)=u_{0}(x)\notag
\end{align}
The class of solutions to \eqref{schrodinger} is invariant by the scaling
\begin{equation}\label{scaling}
u(t,x)\longrightarrow \lambda^{\frac{2}{p-1}}u(\lambda^{2}t,\lambda x)
\end{equation}
This scaling defines a notion of criticality, specifically, for a given Banach space of initial data $u_{0}$, the problem is called critical 
if the norm is invariant under \eqref{scaling}. The problem is called subcritical if the norm of the rescaled solution diverges as $\lambda\rightarrow\infty$; 
if the norm shrinks to zero, then the problem is supercritical.
Moreover, considering the initial value problem \eqref{schrodinger} for $u_{0}\in \dot{H}^{s}(\mathbb{R}^{n})$, the problem is critical when 
$s=s_{c}:=\frac{d}{2}-\frac{2}{p-1}$, subcritical when $s>s_{c}$, and supercritical when $s<s_{c}$.\newline
Now, denote by 
\begin{equation}\label{energyandmass}
M(u)=\int_{\Omega}|u|^{2}dx\;\mbox{and}\;E(u)=\frac{1}{2}\int_{\Omega}|\nabla u|^{2}dx\pm\frac{1}{p+1}\int_{\Omega}|u|^{p+1}dx, 
\end{equation}
the mass and the energy which are conserved.
\vspace{3mm}

For the case of 3D exterior domains, Planchon and Vega obtained in \cite{PV09} an $L^{4}_{t,x}$ 
Strichartz estimate and they used it along with local smoothing estimates 
near the boundary to prove the local well-posedness of the family of nonlinear 
equations \eqref{schrodinger} for $1<p<5$ and $u_{0}\in H_{0}^{1}(\Omega)$, and that the solution is global for the defocusing case (+ sign in \eqref{schrodinger}). 
They also proved scattering for the cubic defocusing nonlinear equation outside star-shaped obstacles for initial data in $H_{0}^{1}$. 
For the energy critical case $p=5$, Ivanovici proved in \cite{I10} local well-posedness for solutions with initial data in $H^{1}$ and global well-posedness for 
small data, outside strictly convex obstacles using the Melrose-Taylor parametrix. Scattering results were also obtained for all subquintic defocusing 
nonlinearities. 
Ivanovici and Planchon then extended in \cite{IP10} the local well
posedness (and global for small energy data) to the quintic nonlinear Schr\"{o}dinger equation for any non-trapping domain in $\mathbb{R}^{3}$ 
using the smoothing effect in $L^{5}_{x}(L^{2}_{t})$ for the linear 
equation. Their local result also holds for the Neumann boundary condition. They also extended the scattering of solutions to the 
defocusing nonlinear equation outside star-shaped obstacles with initial data in $H_{0}^{1}$ for $3\leq p<5$. 
A very recent result was obtained by Killip, Visan, and Zhang in \cite{KVZ12} for the quintic defocusing NLS in the exterior of strictly convex 3D obstacles 
with the Dirichlet boundary condition, where they proved global well-posedness and scattering for all initial data in the energy space.


Our main interest here is exterior domains in 2 dimensions which is known to be the most difficult one regarding scattering questions even in the case of the full space $\mathbb{R}^{n}$. 
In fact, after the 
results of Ginibre and Velo \cite{GV85} for $\mathbb{R}^{n}$ ($n\geq 3$) for the $H^{1}$ subcritical case that corresponds to the case $0<s_{c}<1$, the obstruction of the 
dimension was removed by Nakanishi \cite{N99} (in dimensions 1 and 2, all powers $p$ have an $s_{c}$ that is less than 1), but 
his techniques are not well suited for the domains case. However, a fundamental contribution to the existence and scattering theory in the whole space 
and that turned out later \cite{PV09} to be suitable for the case of exterior domains, was by Colliander, Keel, 
Staffilani, Takaoka, and Tao (\cite{CKSTT04}, \cite{CKSTT08}) through introducing the Morawetz interactive inequalities. Similar problem with low 
dimensions appears due to the sign of the bilaplacian term that comes from the use of a convex weight which is the euclidean distance. 
The sign turns out to be wrong for dimensions less than 3. 
This obstruction was then overcome simultaneously and independently by Colliander, Grillakis, and Tzirakis in \cite{CGT09} as well as by Planchon and Vega in \cite{PV09}. 

In \cite{PV09} the authors also used the bilinear multiplier technique to obtain their results for exterior domains in 3D. Again, just like in the whole space, the obstruction of the dimension appears as a result of the sign of the bilaplacian. 
That is why the local smoothing (Prop. 2.7 in \cite{PV09}), which is a key 
ingredient in the proof of existence and scattering, was given in dimension 3 and higher. 
However, Planchon and Vega recently removed this restriction in \cite{PV2D} and they obtained global existence and scattering results in 2D domains 
exterior to star-shaped obstacles to the nonlinear defocusing problem with initial data in $H_{0}^{1}$ and for $p\geq 5$.

The main idea in \cite{PV2D}
was using the tensor product technique (as developed e.g. in
\cite{CHVZ08} to obtain a quadrilinear Morawetz interaction estimate
in $\mathbb{R}$)
by constructing $v(x,y)=u(x)u(y)$ solution of the nonlinear 
Schr\"{o}dinger in $\Omega\times\Omega$, and then using the local smoothing 
inequality obtained from Morawetz's multipliers in dimension $n=4$ thus resolving the issue of the wrong sign of the bilaplacian in dimension 2. 
Their local smoothing estimate is a key step to get that $D^{1/2}(|u|^{2})$ is in $L^{2}_{t,x}$ for both the nonlinear and linear solutions 
which leads to obtain the global in time Strichartz estimate $L^{p-1}_{t}L^{\infty}_{x}$ (for the case of star-shaped obstacles) which is the key factor to get their 
result.

In this paper we extend the result of Planchon and Vega in two directions, the range of nonlinearities and the class of obstacles under 
consideration. First, we extend the local existence for $p>4$ and for any non-trapping obstacle by using the following set of Strichartz 
estimates obtained by Blair, Smith, and Sogge in \cite{BSS10}:
\begin{theorem}\label{BSStheorem}(Theorem 1.1, \cite{BSS10})
Let $\Omega=\mathbb{R}^{n}\setminus V$ be the exterior domain to a compact non-trapping obstacle with smooth boundary, and $\Delta$ the 
standard Laplace operator on $\Omega$, subject to either Dirichlet or Neumann conditions. Suppose that $p>2$ and $q<\infty$ satisfy
\begin{equation*}
\left\{
\begin{aligned}
&\frac{3}{p}+\frac{2}{q}\leq 1,\;\;n=2,\\
&\frac{1}{p}+\frac{1}{q}\leq\frac{1}{2},\;\;n\geq 3.\\
\end{aligned}
\right.
\end{equation*}
Then for $e^{it\Delta}f$ solution to the linear Schr\"{o}dinger equation with initial data $f$, the following estimates hold
$$\|e^{it\Delta}f\|_{L^{p}([0,T];L^{q}(\Omega))}\leq C\|f\|_{H^{s}(\Omega)},$$
provided that 
$$\frac{2}{p}+\frac{n}{q}=\frac{n}{2}-s.$$
For Dirichlet boundary conditions, the estimates hold with $T=\infty$.
\end{theorem}
\begin{rem}
 Remark that as an application to the nonlinear Schr\"{o}dinger equation in 3D exterior domains, the authors used their above result and interpolation 
to establish the $L^{4}_{t}L^{\infty}_{x}$ Strichartz estimate and present a simple proof to the well-posedness result for small energy data 
to the quintic nonlinear Schr\"{o}dinger equation, a result first obtained by Ivanovici and Planchon \cite{IP10}.
\end{rem}

We will use in this paper the Besov spaces which are defined here using the spectral localization associated to the domain. 
We refer to \cite{IPheat} for a detailed discussion 
and references, and we provide only basic definitions here. Let $\psi(\cdot)\in C_{0}^{\infty}(\mathbb{R}\setminus\{0\})$ and 
$\psi_{j}(\cdot)=\psi(2^{-2j}\cdot)$. On the domain $\Omega$, one has the spectral resolution of the Dirichlet Laplacian, and we may define 
smooth spectral projections $\Delta_{j}=\psi_{j}(-\Delta_{D})$ as continuous operators on $L^{2}$ (they are also 
continuous on $L^{p}$ for all $p$). Moreover, just like the whole space case, these projections obey Bernstein estimates.  
\begin{definition}
 Let $f\in\mathit{S}^{'}(\Omega)$ and let $\Delta_{j}=\psi(-2^{-2j}\Delta_{D})$ be a spectral localization with respect to the Dirichlet 
Laplacian $\Delta_{D}$ such that $\sum_{j}\Delta_{j}=Id$. 
We say $f$ belongs to $\dot{B}_{p}^{s,q}(\Omega)$
($s\in\mathbb{R}$, $1\leq p,q\leq+\infty$) if
$$\left(2^{js}\|\Delta_{j}f\|_{L^{p}}\right)\in l^{q},$$
and $\sum_{j}\Delta_{j}f$ converges to f in $\mathit{S}^{'}$.
\end{definition}
Note that $\dot B^{1,2}_{2}=\dot H^{1}_{0}$ and by analogy we set
$\dot H^s$ to be just $\dot B^{s,2}_{2}$. 
The Banach space $\dot{B}_{p}^{s,q}$ is equipped with following norm:
$$\|f\|_{\dot{B}_{p}^{s,q}}:=\left(\sum_{j\in\mathbb{Z}}\|2^{js}\Delta_{j}f\|_{L^{p}}^{q}\right)^{\frac{1}{q}}.$$ 
\begin{rem} In our range of interest, this intrinsic definition may be proved to be equivalent
  with the more well-known definition using the restriction to the
  domain $\Omega$ of functions in $\dot
  B^{s,q}_p(\mathbb{R}^n)$. However, we
  will not need this equivalence.
\end{rem}

We first obtain the following result:
\begin{theorem}\label{main result 1}
Let $\Omega$ be $\mathbb{R}^{2}\setminus V$, where $V$ is a non-trapping obstacle, and $u_{0}\in \dot{B}^{s_{c},1}_{2}(\Omega)$. Then, there exists
$T(u_{0})$ such that the nonlinear equation:
\begin{equation*}
\left\{
\begin{aligned}
&i\partial_{t}u+\Delta u=\pm|u|^{p-1}u\;\;x\in\Omega,\;t\in\mathbb{R},\;p> 4\\
&u|_{\mathbb{R}\times\partial\Omega}=0\\
&u(0,x)=u_{0}(x),
\end{aligned}
\right.
\end{equation*}
admits a unique solution $u$ in the function space 
$$C([0,T];\dot{B}^{s_{c},1}_{2}(\Omega))\cap L^{p-1}([0,T];L^{\infty}(\Omega)).$$ Moreover, if $u_{0}\in H_{0}^{1}(\Omega)$, then the 
solution stays in $H_{0}^{1}(\Omega)$ and it is global in time for the defocusing equation.
\end{theorem}

Then, we prove the scattering for the defocusing equation with initial data in $H_{0}^{1}(\Omega)$ 
for a class of almost star-shaped obstacles 
satisfying the following geometric condition: Given $0<\epsilon\leq1$
 \begin{equation}\label{The geometric-condition}
(x_{1},\epsilon x_{2})\cdot n_{x} > 0\;\;\mbox{for}\;\;x=(x_{1},x_{2})\in\partial V 
\end{equation}
where $n_{x}$ is the exterior unit normal to $\partial V$.

\begin{rem}
 In fact, for $\epsilon=1$, which corresponds to the star-shaped case, we don't need the strictness in \eqref{The geometric-condition} 
(see \cite{PV2D}).
\end{rem}

Almost star-shaped obstacles that are a natural generalization of the star-shaped were introduced by Ivrii in \cite{Ivrii} in the setting of local energy decay for the linear wave equation. 
In section \ref{Geometry}, we provide an explicit definition for such obstacles as well as an interpretation of the geometric condition 
\eqref{The geometric-condition}.\vspace{2mm}\newline
We obtain the following theorem:

\begin{theorem}\label{mainresult}
Let $\Omega$ be $\mathbb{R}^{2}\setminus V$, where $V$ is an almost star-shaped obstacle satisfying the condition \eqref{The geometric-condition}, 
and $u_{0}\in H_{0}^{1}(\Omega)$. Then the global solution for the defocusing 
equation
\begin{equation*}
\left\{
\begin{aligned}
&i\partial_{t}u+\Delta u=|u|^{p-1}u\;\;x\in\Omega,\;t\in\mathbb{R},\;p> 4\\
&u|_{\mathbb{R}\times\partial\Omega}=0\\
&u(0,x)=u_{0}(x),
\end{aligned}
\right.
\end{equation*}
scatters in $H_{0}^{1}$. 
\end{theorem}
{\bf Acknowledgements}. I would like to thank Fabrice Planchon for suggesting the problem and commenting on the manuscript.

\section{Proof of the local and global existence (Theorem \ref{main result 1})}
We want to solve 
\begin{align}\label{p4schrodinger}
&i\partial_{t}u+\Delta u=\pm|u|^{p-1}u\;\;in\;\;\Omega=\mathbb{R}^{2}\setminus V,\;\;p>4\notag\\
&u|_{\mathbb{R}\times\partial\Omega}=0\\
&u(0,x)=u_{0}(x)\notag
\end{align}
We will set $p=\frac{3}{1-\epsilon_{0}}+1$ with $0<\epsilon_{0}<1$.\newline\newline
Note that the Sobolev space with the invariant norm under the scaling \eqref{scaling} 
is $\dot{H}^{s_{c}}$ with $s_{c}=\frac{1}{3}+\frac{2\epsilon_{0}}{3}$.\newline
Using the estimate obtained by Blair, Smith, and Sogge (Theorem \ref{BSStheorem}), we can obtain another linear estimate in the Besov space 
$\dot{B}^{s_{c},1}_{2}$. This is stated in the following proposition:

 \begin{proposition}\label{Lp-1Linfty}
    Let $\Omega=\mathbb{R}^{2}\setminus V$, where $V$ is a
    non-trapping obstacle with smooth boundary, and $\Delta$ is the
    Dirichlet Laplacian.  Then for $e^{it\Delta}f$ solution to the
    linear Schr\"{o}dinger equation with initial data $f$, we have
    \begin{equation}\label{linearestimate}
      \|e^{it\Delta}f\|_{L^{\frac{3}{1-\epsilon_{0}}}([0,+\infty];L^{\infty}(\Omega)}\lesssim\|f\|_{\dot{B}_{2}^{s_{c},1}(\Omega)}
    \end{equation}
  \end{proposition}
  \begin{proof}
    For exterior domains in $\mathbb{R}^{2}$ and given any
    $0<\epsilon<1$, we have the following Strichartz estimate obtained
    by Blair, Smith, and Sogge
    \begin{equation}\label{B-S-S}
      \|e^{it\Delta}f\|_{L^{\frac{3}{1-\epsilon}}_{t}L^{\frac{2}{\epsilon}}_{x}}\leq C(\epsilon)\|f\|_{\dot{H}^{\frac{1}{3}(1-\epsilon)}}
    \end{equation}
    On a dyadic block $\Delta_{j}f$, where $\Delta_{j}$ is defined via
    the Dirichlet Laplacian $\Delta$, the Blair-Smith-Sogge estimate
    is written as follows
    \begin{equation}\label{BSSdyadic}
      \|\Delta_{j}(e^{it\Delta}f)\|_{L^{\frac{3}{1-\epsilon}}_{t}L^{\frac{2}{\epsilon}}_{x}}\lesssim 2^{j\frac{1-\epsilon}{3}}\|\Delta_{j}f\|_{L^{2}}
    \end{equation}
    for any $0<\epsilon<1$.  This can be easily obtained using
    \eqref{B-S-S} and the fact that $\Delta_{j}$ commutes with
    $e^{it\Delta}$ as well as a Bernstein's inequality.\newline Now,
    we choose $\epsilon=\epsilon_{0}$, we have
$$2^{\epsilon_{0}j}\|\Delta_{j}(e^{it\Delta}f)\|_{L^{\frac{3}{1-\epsilon_{0}}}_{t}L^{\frac{2}{\epsilon_{0}}}_{x}}\lesssim 2^{j\frac{1+2\epsilon_{0}}{3}}\|\Delta_{j}f\|_{L^{2}}$$
But by Bernstein we have,
$$\|\Delta_{j}(e^{it\Delta}f)\|_{L^{\infty}_{x}}\lesssim2^{j\epsilon_{0}}\|\Delta_{j}(e^{it\Delta}f)\|_{L^{\frac{2}{\epsilon_{0}}}_{x}}$$
hence
\begin{align*}
  \|e^{it\Delta}f\|_{L^{\frac{3}{1-\epsilon_{0}}}_{t}L^{\infty}_{x}}&\leq\sum_{j}\|\Delta_{j}(e^{it\Delta}f)\|_{L^{\frac{3}{1-\epsilon_{0}}}_{t}L^{\infty}_{x}}\\
  &\lesssim\sum_{j}2^{j\epsilon_{0}}\|\Delta_{j}(e^{it\Delta}f)\|_{L^{\frac{3}{1-\epsilon_{0}}}_{t}L^{\frac{2}{\epsilon_{0}}}_{x}}\\
  &\lesssim
  \sum_{j}2^{j\frac{1+2\epsilon_{0}}{3}}\|\Delta_{j}f\|_{L^{2}}\left(=\left\|f\right\|_{\dot{B}_{2}^{s_{c},1}}\right)
\end{align*}
Hence we get the following linear estimate
$$\left\|e^{it\Delta}f\right\|_{L^{\frac{3}{1-\epsilon_{0}}}_{t}L^{\infty}_{x}}\lesssim\|f\|_{\dot{B}_{2}^{s_{c},1}}$$
which ends the proof of Proposition \ref{Lp-1Linfty}.
\end{proof}

Now, using the estimate \eqref{linearestimate}, we can solve the
nonlinear equation \eqref{p4schrodinger} with initial data in
$\dot{B}_{2}^{s_{c},1}$ locally in time in the function space $E_{T}$
given by: for $T>0$
$$E_{T}=C([0,T];\dot{B}_{2}^{s_{c},1}(\Omega))\cap L^{\frac{3}{1-\epsilon_{0}}}([0,T];L^{\infty}(\Omega)).$$ 

Set $F(x)=|x|^{\frac{3}{1-\epsilon_{0}}}x$ (or
$-|x|^{\frac{3}{1-\epsilon_{0}}}x$ in the focusing case) and choose
$T$ small enough so that
$\|e^{it\Delta}u_{0}\|_{L_{[0,T]}^{\frac{3}{1-\epsilon_{0}}}L^{\infty}_{x}}<c$
for a small constant $c$ to be determined and which is linked to the size of the Besov norm of $u_0$. The larger the
latter is, the smaller the former will have to be.

\begin{rem}
 Remark that the smallness of this quantity can be made explicit if $u_{0}$ 
is in $H^{1}$ (not just $\dot{B}_{2}^{s_{c},1}$), and then $T$ will be like an inverse power of the norm $\dot{H}^{1}$ of $u_{0}$ (see for example page 22 of 
\cite{BP04} for a similar reasoning). Moreover, for the defocusing case,
the $H^{1}$ norm is controlled and thus the local time of existence is uniform and one can consequently iterate the local existence result to a global 
result.
\end{rem}

We define the following mapping for $w\in E_{T}$
$$\phi(w)(t):=\int_{s<t}e^{i(t-s)\Delta}F(e^{is\Delta}u_{0}+w(s))ds$$
then we have
\begin{align}\label{BtoB}
\|\phi(w)\|_{E_{T}}&\lesssim\|F(e^{it\Delta}u_{0}+w)\|_{L^{1}([0,T];\dot{B}^{s_{c},1}_{2})}\\
&\lesssim\|e^{it\Delta}u_{0}+w\|_{L^{\infty}_{T}\dot{B}_{2}^{s_{c},1}}\|e^{it\Delta}u_{0}+w\|_{{ L^{\frac{3}{1-\epsilon_{0}}}_{T}L^{\infty}_{x}}}^{\frac{3}{1-\epsilon_{0}}}\notag
\end{align}
The first part can be shown using the linear estimate \eqref{linearestimate}, as for the second part, it is due to the following lemma (for the 
special case $f=e^{it\Delta}u_{0}+w$ and $g=0$): 
\begin{lemma}\label{besov}
Consider $f,g\in L^{\infty}_{T}\dot{B}_{p}^{s,q}\cap L^{\alpha-1}_{T}L^{\infty}_{x}$ with $0<s<2$, 
then if $F(x)=|x|^{\alpha-1}x$ (or $|x|^{\alpha}$) and $\alpha\geq 3$ we 
have
\begin{align*}
\|F(f)-&F(g)\|_{L^{1}_{T}\dot{B}^{s,q}_{p}}\lesssim\|f-g\|_{L^{\infty}_{T}\dot{B}_{p}^{s,q}}(\|f\|^{\alpha-1}_{ L^{\alpha-1}_{T}L^{\infty}_{x}}+\|g\|^{\alpha-1}_{ L^{\alpha-1}_{T}L^{\infty}_{x}})\\
&+\|f-g\|_{ L^{\alpha-1}_{T}L^{\infty}_{x}}(\|f\|_{L^{\infty}_{T}\dot{B}_{p}^{s,q}}\|f\|_{ L^{\alpha-1}_{T}L^{\infty}_{x}}^{\alpha-2}+
\|g\|_{L^{\infty}_{T}\dot{B}_{p}^{s,q}}\|g\|_{L^{\alpha-1}_{T}L^{\infty}_{x}}^{\alpha-2})
\end{align*}
\end{lemma}
\begin{proof}
 This lemma can be proved by writing 
$$F(f)-F(g)=(f-g)\int_{0}^{1}F^{'}(\theta f+(1-\theta)g)d\theta,$$
and splitting this difference into two paraproducts. For a detailed proof, we refer to Lemma 4.10 in \cite{IP10} which is given for functions in $\dot{B}_{p}^{s,q}\cap L^{r}$. In fact, we are considering a special case of that 
lemma with $r=\infty$. The time norms are harmless and can be easily inserted using H\"{o}lder. 
Note that such a result is by now classical if the domain is just $\mathbb{R}^{n}$, 
and where the easiest path to prove it is to use the characterization 
of Besov spaces using finite differences. By contrast, on domains, \cite{IP10} 
provides a direct proof using paraproducts which are based on the
spectral localization.
\end{proof}
\vspace{3mm}
Choosing the small constant $c$ such that $c^{\frac{3}{1-\epsilon_{0}}}\|u_{0}\|_{\dot{B}^{s_{c},1}_{2}}<<1$, the estimate 
\eqref{BtoB} shows that one can have a small ball in $w$ of $E_{T}$ that maps into itself. 
A similar argument on $\|\phi(w)-\phi(w')\|_{E_{T}}$ for $w'\in E_{T}$ shows that $\phi$ is a contraction on the small ball: 
by Lemma \ref{besov} (with $\alpha=\frac{3}{1-\epsilon_{0}}+1$), if $u=e^{it\Delta}u_{0}+w$ and $v=e^{it\Delta}u_{0}+w'$
\begin{align*}
\|\phi(w)-\phi(w')\|_{E_{T}}&\lesssim\|F(u)-F(v)\|_{L^{1}([0,T];\dot{B}^{s_{c},1}_{2})}\\
&\lesssim\|w-w'\|_{L^{\infty}_{T}\dot{B}_{2}^{s_{c},1}}(\|u\|^{\alpha-1}_{ L^{\alpha-1}_{T}L^{\infty}_{x}}+\|v\|^{\alpha-1}_{ L^{\alpha-1}_{T}L^{\infty}_{x}})\\
&+\|w-w'\|_{ L^{\alpha-1}_{T}L^{\infty}_{x}}(\|u\|_{L^{\infty}_{T}\dot{B}_{2}^{s_{c},1}}\|u\|_{ L^{\alpha-1}_{T}L^{\infty}_{x}}^{\alpha-2}+
\|v\|_{L^{\infty}_{T}\dot{B}_{2}^{s_{c},1}}\|v\|_{L^{\alpha-1}_{T}L^{\infty}_{x}}^{\alpha-2})
\end{align*}
Note that the smallness comes from the $||\cdot||^{\alpha-k}$ factors,
with $k=1,2$. 
Hence, by the fixed point theorem, there exists a unique $w$ in the small ball such that $\phi(w)=w$ and thus $u$ set as $u=e^{it\Delta}u_{0}+w$ is a solution 
to the nonlinear Schr\"{o}dinger equation \eqref{p4schrodinger} that satisfies 
\begin{equation}\label{duhamel}
u=e^{it\Delta}u_{0}+\int_{s<t}e^{i(t-s)\Delta}F(u(s))ds.
\end{equation}
\newline
Now, we will show that if the initial data $u_{0}\in H_{0}^{1}$, then the solution $u$ remains in $H_{0}^{1}$. In fact, if $u_{0}\in H_{0}^{1}$ then 
$u_{0}\in L^{2}=\dot{B}_{2}^{0,2}$ and $u_{0}\in \dot{H}^{1}=\dot{B}_{2}^{1,2}$ (from now on $\dot{H}^{1}$ will always correspond to 
$\dot{H}^{1}_{0}$). 
Using the following interpolation inequality
$$\|u_{0}\|_{\dot{B}_{2}^{s_{c},1}}\lesssim\|u_{0}\|^{s_{c}}_{\dot{B}_{2}^{1,\infty}}\|u_{0}\|^{1-s_{c}}_{\dot{B}_{2}^{0,\infty}}$$
and the fact that 
$$\|u_{0}\|_{\dot{B}_{2}^{1,\infty}}\leq\|u_{0}\|_{\dot{B}_{2}^{1,2}}$$
and 
$$\|u_{0}\|_{\dot{B}_{2}^{0,\infty}}\leq\|u_{0}\|_{\dot{B}_{2}^{0,2}}$$ 
we get that
$$\|u_{0}\|_{\dot{B}_{2}^{s_{c},1}}\lesssim\|u_{0}\|_{\dot{H}^{1}}^{s_{c}}\|u_{0}\|^{1-s_{c}}_{L^{2}}$$
Thus $u_{0}\in\dot{B}^{s_{c},1}_{2}$ and the nonlinear equation \eqref{p4schrodinger} with initial data $u_{0}\in H_{0}^{1}(\Omega)$ has a local solution in $E_{T}$ given by the Duhamel formula \eqref{duhamel}. Hence, we have
$$\|u\|_{C_{T}\dot{H}^{1}}\leq\|u_{0}\|_{\dot{H}^{1}}+\||u|^{\frac{3}{1-\epsilon_{0}}}u\|_{L^{1}_{T}\dot{H}^{1}}\leq\|u_{0}\|_{\dot{H}^{1}}+\|u\|^{\frac{3}{1-\epsilon_{0}}}_{L_{T}^{\frac{3}{1-\epsilon_{0}}}L^{\infty}_{x}}\|u\|_{L^{\infty}_{T}\dot{H}^{1}}$$
where the nonlinearity is again dealt with by Lemma \ref{besov} (with $s=1$, $p=q=2$). We also have 
$$\|u\|_{C_{T}L^{2}_{x}}\leq\|u_{0}\|_{L^{2}_{x}}+\|u\|^{\frac{3}{1-\epsilon_{0}}}_{L_{T}^{\frac{3}{1-\epsilon_{0}}}L^{\infty}_{x}}\|u\|_{L^{\infty}_{T}L^{2}_{x}}.$$
As the solution $u$ is constructed such that its $L_{T}^{\frac{3}{1-\epsilon_{0}}}L^{\infty}_{x}$ norm is sufficiently small, the above inequalities yield that $u\in C([0,T];H_{0}^{1})$.

\section{Scattering for the defocusing equation (Proof of Theorem \ref{mainresult})}
In this section, we will show that for the defocusing case with initial data in $H_{0}^{1}$ and for domains $\Omega$ exterior to star-shaped obstacles as well as 
for a class of almost star-shaped obstacles (see section \ref{Geometry}), the solution to the nonlinear equation scatters in $H_{0}^{1}$. To prove that is suffices to show that given any interval $I$ of time where the solution exists the $L_{I}^{\frac{3}{1-\epsilon_{0}}}L^{\infty}_{x}$ norm is controlled 
by a universal constant that is independent $I$. To achieve this we will use the conservation laws of the mass and energy \eqref{energyandmass}, 
as well as additional space-time control of the solution.

\subsection{The case of star-shaped obstacles}\label{star-shaped case}
For star-shaped obstacles, in addition to the conservation laws of the mass and energy, we will use the fact that 
the $L^{4}_{t}L^{8}_{x}$ norm is controlled,
which is a consequence of the following result by Planchon and Vega \cite{PV2D}:
\begin{proposition}(Planchon-Vega, \cite{PV2D})\label{L4L8}
 Let $\Omega$ be $\mathbb{R}^{2}\setminus V$, where $V$ is a star-shaped and bounded domain. Then $u$ the solution of
\begin{equation*}
\left\{
\begin{aligned}
&i\partial_{t}u+\Delta u=|u|^{p-1}u\;\;p\geq 1\\
&u|_{\mathbb{R}\times\partial\Omega}=0
\end{aligned}
\right.
\end{equation*}
satisfies
$$\|D^{1/2}(|u|^{2})\|_{L^{2}_{t}L^{2}_{x}}\lesssim M^{3/4}E^{1/4},$$
where $u$ is extended by zero for $x\not\in\Omega$ to make sense of
the half-derivative operator.
\end{proposition}
\begin{rem}
 Remark that this result is also true for the linear equation, and it plays the key role in proving the $L^{p-1}_{t}L^{\infty}_{x}$ 
(with $p-1\geq 4$) Strichartz estimate for star-shaped obstacles in their paper. This is what restricted the range of $p$ in \cite{PV2D}, whereas the result of 
Blair, Smith, and Sogge (Theorem \ref{BSStheorem}) allows us to get that estimate with a $p>4$. 
\end{rem}
This proposition combined with a Sobolev embedding yields that 
$$\|u\|_{L^{4}_{I}L^{8}_{x}}\lesssim M^{3/8}E^{1/8}.$$

Hence we now know that the solution $u$ to the defocusing equation exterior to star-shaped obstacles is such that
$$u\in L^{4}_{I}L^{8}_{x}\cap L^{\infty}_{I}\dot{H}^{1}$$

But, using the fact that $L^{8}$ is continuously included in $\dot{B}^{0,8}_{8}$ and $\dot{H}^{1}=\dot{B}^{1,2}_{2}$, as well as 
the following inequalities for Besov spaces:
$$q_{1}\leq q_{2}\Rightarrow \|u\|_{\dot{B}^{s,q_{2}}_{p}}\leq\|u\|_{\dot{B}^{s,q_{1}}_{p}}$$

$$p_{1}\leq p_{2}\Rightarrow \|u\|_{\dot{B}^{s-n(\frac{1}{p_{1}}-\frac{1}{p_{2}}),q}_{p_{2}}}\lesssim\|u\|_{\dot{B}^{s,q}_{p_{1}}}$$
we get the following continuous embeddings:
$$L^{8}_{x}\subset\dot{B}^{-1/4,\infty}_{\infty}\;\mbox{and}\;\dot{H}^{1}\subset\dot{B}^{0,\infty}_{\infty}$$

So, the solution $u$ is such that 
$$u\in L^{4}_{I}(\dot{B}^{-1/4,\infty}_{\infty})\cap L^{\infty}_{I}(\dot{B}^{0,\infty}_{\infty})$$

thus using the well known interpolation inequalities for Lebesgue and Besov spaces, we get that

\begin{equation*}
u\in L^{q}_{I}(\dot{B}^{\gamma,\infty}_{\infty}) 
\end{equation*}
with
$$\frac{1}{q}=\frac{\alpha}{4}+\frac{1-\alpha}{\infty}=\frac{\alpha}{4}$$
and
$$\gamma=\frac{-\alpha}{4}+0\times(1-\alpha)=\frac{-\alpha}{4}$$
for any $\alpha\in]0,1[$. We conveniently choose $\alpha=\frac{8}{9}(1-\epsilon_{0})$ (based on the scaling of the space $L^{\frac{3}{1-\epsilon_{0}}}_{T}L^{\infty}_{x}$), 
and get that $u\in L^{\frac{9}{2(1-\epsilon_{0})}}_{I}\left(\dot{B}^{-\frac{2}{9}(1-\epsilon_{0}),\infty}_{\infty}\right)$, and
\begin{align}\label{secondbesov}
\|u\|_{L^{\frac{9}{2(1-\epsilon_{0})}}_{I}\dot{B}^{-\frac{2}{9}(1-\epsilon_{0}),\infty}_{\infty}}&\lesssim\|u\|^{\alpha}_{L^{4}_{I}(\dot{B}^{-1/4,\infty}_{\infty})}\|u\|^{1-\alpha}_{L^{\infty}_{I}(\dot{B}^{0,\infty}_{\infty})}\notag\\
&\lesssim\|u\|^{\alpha}_{ L^{4}_{I}L^{8}_{x}}\|u\|^{1-\alpha}_{L^{\infty}_{I}\dot{H}^{1}}\leq C(M,E)
\end{align}

Now, given any interval $I$ of time where the solution exists, and given any $\eta>0$ there is a finite number of disjoint intervals $I_{1},\cdots I_{N}$ 
such that $\displaystyle\bigcup_{j=1}^{N}I_{j}=I$ with $N=N(\eta)$ and 
$$\|u\|_{L^{\frac{9}{2(1-\epsilon_{0})}}_{I_{j}}\dot{B}^{-\frac{2}{9}(1-\epsilon_{0}),\infty}_{\infty}}=\eta,\;\;j<N(\eta),$$
$$\|u\|_{L^{\frac{9}{2(1-\epsilon_{0})}}_{I_{j}}\dot{B}^{-\frac{2}{9}(1-\epsilon_{0}),\infty}_{\infty}}\leq\eta,\;\;j=N(\eta).$$

Hence, due to \eqref{secondbesov}, $$N(\eta)\lesssim C(M,E)\eta^{-1}.$$
Now, we fix an $\epsilon$ (to be chosen later) such that $0<\epsilon<\epsilon_{0}$ and we introduce the following lemma:
\begin{lemma}
Let $\Omega=\mathbb{R}^{2}\setminus V$, where $V$ is a non-trapping obstacle with smooth boundary, and $\Delta$ is the Dirichlet Laplacian. 
Then for $e^{it\Delta}f$ solution to the linear Schr\"{o}dinger equation with initial data $f$, we have
\begin{equation}\label{anotherstrichartz}
\|e^{it\Delta}f\|_{L^{\frac{3}{1-\epsilon}}_{t}\dot{B}_{\infty}^{\frac{2(\epsilon_{0}-\epsilon)}{3},1}}\lesssim\|f\|_{\dot{B}^{s_{c},1}_{2}}
\end{equation}
\end{lemma}
\begin{proof}
To prove this we will use again the Blair-Smith-Sogge estimate on a dyadic block $\Delta_{j}f$:
\begin{equation*}
 \|\Delta_{j}(e^{it\Delta}f)\|_{L^{\frac{3}{1-\epsilon}}_{t}L^{\frac{2}{\epsilon}}_{x}}\lesssim 2^{j\frac{1-\epsilon}{3}}\|\Delta_{j}f\|_{L^{2}}
\end{equation*}
thus
$$2^{\frac{2\epsilon_{0}+\epsilon}{3}j}\|\Delta_{j}(e^{it\Delta}f)\|_{L^{\frac{3}{1-\epsilon}}_{t}L^{\frac{2}{\epsilon}}_{x}}\lesssim 2^{j\frac{1+2\epsilon_{0}}{3}}\|\Delta_{j}f\|_{L^{2}}$$
But by Bernstein we have,
$$\|\Delta_{j}(e^{it\Delta}f)\|_{L^{\infty}_{x}}\lesssim2^{j\epsilon}\|\Delta_{j}(e^{it\Delta}f)\|_{L^{\frac{2}{\epsilon}}_{x}}$$
hence
$$2^{\frac{2(\epsilon_{0}-\epsilon)}{3}j}\|\Delta_{j}(e^{it\Delta}f)\|_{L_{t}^{\frac{3}{1-\epsilon}}L^{\infty}_{x}}
\lesssim 2^{\frac{2\epsilon_{0}+\epsilon}{3}j}\|\Delta_{j}(e^{it\Delta}f)\|_{L_{t}^{\frac{3}{1-\epsilon}}L^{\frac{2}{\epsilon}}_{x}}
\lesssim 2^{\frac{1+2\epsilon_{0}}{3}j}\|\Delta_{j}f\|_{L^{2}_{x}}$$
and thus get
\begin{equation*}
\|e^{it\Delta}f\|_{L^{\frac{3}{1-\epsilon}}_{t}\dot{B}_{\infty}^{\frac{2(\epsilon_{0}-\epsilon)}{3},1}}\leq\sum_{j}2^{\frac{2(\epsilon_{0}-\epsilon)}{3}j}\|\Delta_{j}(e^{it\Delta}f)\|_{L_{t}^{\frac{3}{1-\epsilon}}L^{\infty}_{x}}
\lesssim\|f\|_{\dot{B}^{s_{c},1}_{2}}
\end{equation*}
\end{proof}
Using the Duhamel formula \eqref{duhamel} and the above estimate \eqref{anotherstrichartz} shows that the solution we constructed locally is also in 
$L^{\frac{3}{1-\epsilon}}_{I}\dot{B}_{\infty}^{\frac{2(\epsilon_{0}-\epsilon)}{3},1}$, and in particular we have by Duhamel on $J_{j}(t)=[t_{j},t]\subset I_{j}=[t_{j},t_{j+1})$:

 \begin{equation}\label{thesecond}
\|u\|_{L^{\frac{3}{1-\epsilon}}_{J_{j}(t)}\dot{B}^{\frac{2(\epsilon_{0}-\epsilon)}{3},1}_{\infty}}\lesssim\|u(t_{j})\|_{\dot{B}^{s_{c},1}_{2}}
+\|u\|^{\frac{3}{1-\epsilon_{0}}}_{L^{\frac{3}{1-\epsilon_{0}}}_{J_{j}(t)}L^{\infty}_{x}}\|u\|_{L^{\infty}_{J_{j}(t)}\dot{B}^{s_{c},1}_{2}}
\end{equation}

On the other hand, we have the following interpolation inequality
$$\|u\|_{\dot{B}^{0,1}_{\infty}}\lesssim\|u\|^{\beta}_{\dot{B}^{-\frac{2}{9}(1-\epsilon_{0}),\infty}_{\infty}}\|u\|^{1-\beta}_{\dot{B}^{\frac{2}{3}(\epsilon_{0}-\epsilon),\infty}_{\infty}}$$
with $0=-\frac{2}{9}\beta(1-\epsilon_{0})+\frac{2}{3}(1-\beta)(\epsilon_{0}-\epsilon)$. For simplicity, we choose 
$\epsilon=\frac{2\epsilon_{0}}{3}$, and thus $\beta=\epsilon_{0}$. Using the fact that $\dot{B}^{0,1}_{\infty}$ is continuously included in 
$L^{\infty}$ and that $$\|u\|_{\dot{B}^{\frac{2}{9}\epsilon_{0},\infty}_{\infty}}\leq\|u\|_{\dot{B}^{\frac{2}{9}\epsilon_{0},1}_{\infty}}$$
we get that
$$\|u\|_{L^{\infty}_{x}}\lesssim\|u\|^{\epsilon_{0}}_{\dot{B}^{-\frac{2}{9}(1-\epsilon_{0}),\infty}_{\infty}}\|u\|^{1-\epsilon_{0}}_{\dot{B}^{\frac{2}{9}\epsilon_{0},1}_{\infty}}$$
hence,
\begin{equation}\label{thefirst}
\|u\|_{L^{\frac{3}{1-\epsilon_{0}}}_{J_{j}(t)}L^{\infty}_{x}}\lesssim\|u\|^{\epsilon_{0}}_{L^{\frac{9}{2(1-\epsilon_{0})}}_{J_{j}(t)}\dot{B}^{-\frac{2}{9}(1-\epsilon_{0}),\infty}_{\infty}}\|u\|^{1-\epsilon_{0}}_{L^{\frac{3}{1-\epsilon}}_{J_{j}(t)}\dot{B}^{\frac{2}{9}\epsilon_{0},1}_{\infty}}
\end{equation}
Note that $n(t)=\|u\|_{L^{\frac{3}{1-\epsilon_{0}}}_{J_{j}(t)}L^{\infty}_{x}}$ is a continuous function and $n(t_{j})=0$.
Now, since $$\|u\|_{\dot{B}^{s_{c},1}_{2}}\lesssim\|u\|^{s_{c}}_{\dot{H}^{1}}\|u\|_{L^{2}}^{1-s_{c}}\leq K$$ where $K$ is a constant that depends 
on the conserved mass and energy, \eqref{thesecond} and \eqref{thefirst} yield
\begin{align*}
\|u\|_{L^{\frac{3}{1-\epsilon_{0}}}_{J_{j}(t)}L^{\infty}_{x}}&\lesssim\|u\|^{\epsilon_{0}}_{L^{\frac{9}{2(1-\epsilon_{0})}}_{J_{j}(t)}\dot{B}^{-\frac{2}{9}(1-\epsilon_{0}),\infty}_{\infty}}K^{1-\epsilon_{0}}(1
+\|u\|^{\frac{3}{1-\epsilon_{0}}}_{L^{\frac{3}{1-\epsilon_{0}}}_{J_{j}(t)}L^{\infty}_{x}})^{1-\epsilon_{0}}\\
&\lesssim \eta^{\epsilon_{0}}K^{1-\epsilon_{0}}(1+\|u\|^{3}_{L^{\frac{3}{1-\epsilon_{0}}}_{J_{j}(t)}L^{\infty}_{x}})
\end{align*}
choosing $\eta$ such that $\eta^{\epsilon_{0}}K^{1-\epsilon_{0}}$ is small enough, we conclude that 
 $\|u\|_{L^{\frac{3}{1-\epsilon_{0}}}_{I_{j}}L^{\infty}_{x}}$ remains bounded by a universal constant $C_{1}$ independent of the time interval of existence $I$.  
Therefore, $$\|u\|_{L^{\frac{3}{1-\epsilon_{0}}}_{I}L^{\infty}_{x}}\leq C_{1}N\lesssim C(M,E).$$ 
Hence, our global solution satisfies
 $$\|u\|_{L^{\frac{3}{1-\epsilon_{0}}}_{\mathbb{R}}L^{\infty}_{x}}\leq C(M,E).$$
Finally, defining $u_{+}\in H_{0}^{1}$ as 
$$u_{+}=u_{0}+\int_{0}^{\infty}e^{i\tau\Delta}|u|^{p-1}u(\tau)d\tau$$
and similarly for $u_{-}$, we get the scattering 
$$\|u(\cdot,t)-e^{it\Delta}u_{\pm}\|=o(1)\;\;\;t\rightarrow\pm\infty.$$

\subsection{The case of almost star-shaped obstacles}

In this section, we will prove the scattering for the defocusing equation for almost star-shaped obstacles $V$ satisfying the following geometric 
condition:
 Given an $\epsilon$ such that $0<\epsilon<1$,
\begin{equation}\label{geometric-condition}
(x_{1},\epsilon x_{2})\cdot n_{x} > 0\;\;\mbox{for}\;\;x=(x_{1},x_{2})\in\partial V 
\end{equation}
where $n_{x}$ is the exterior unit normal to $\partial V$.\vspace{2mm}\newline
In this case, we lost the $L^{4}_{t}L^{8}_{x}$ control which was obtained under 
the star-shaped assumption. However, we will establish a similar control in some $L^{a}_{t}L^{b}_{x}$ norm that will play the same role in proving 
the scattering.  

\subsubsection{Geometry of the obstacle}\label{Geometry}
In 1969, Ivrii introduced the notion of almost star-shaped obstacles in the setting of the linear wave equation. He proved in \cite{Ivrii} 
local energy decay results for domains exterior to such obstacles in odd dimensions $n>1$. An almost star-shaped obstacle $V$
($\Omega=\mathbb{R}^{n}\setminus V$) is defined as follows:
\begin{definition}\label{Def-Ivrii}
A bounded open region $V$ with a boundary in class $C^{1}$ is said to be almost star-shaped if there exists a D bounded open neighborhood 
of $\overline{V}$, a real-valued function $\phi\in C^{2}(\overline{D}\cap\Omega)$ and a constant $c_{0}$ such that:
\begin{itemize}
 \item $\phi(x)<c_{0}$, $x\in D\cap\Omega$, $\phi(x)=c_{0},$ $x\in\partial D$.
\item $|\nabla \phi(x)|\geq const>0$, $x\in \overline{D}\cap\Omega$.
\item The level surfaces $\phi(x)=c$ are strongly convex; the radius of curvature in all directions at all points of $\Omega\cap\overline{D}$ 
is uniformly bounded from above.
\item At points of intersection of the level surfaces with $\partial V$ their outer normals and the outer normal to $\partial V$ form an 
angle which is not greater than a right angle.  
\end{itemize}
\end{definition}
These obstacles are a natural generalization of the star-shaped obstacles. If the level surfaces are spheres with a common center, then $V$
is star-shaped and conversely.
According to the above definition, an almost star-shaped obstacle with ellipses  as level surfaces satisfies the 
geometric condition \eqref{geometric-condition}, where the strict inequality corresponds to an angle
strictly less than a right angle in the 4th condition of Definition \ref{Def-Ivrii}.
More explicitly, the function $\phi$ is given by 
$\phi(x)=\sqrt{x_{1}^{2}+\epsilon x_{2}^{2}}$ and this corresponds to what is called the gauge function of the convex body delimited by 
the ellipse given by the equation $x_{1}^{2}+\epsilon x_{2}^{2}=c^{2}$.

 We also remark that the case of almost star-shaped obstacles corresponds to the works of Strauss \cite{S75} and Morawetz \cite{M75} that followed 
in 1975 (independently of Ivrii's work which was unknown to them at that time) on local energy decay for the linear wave equation. Moreover, in the 
same setting and around the same time in the 70's, 
another generalization to the star-shaped case was introduced which is the illuminating geometry. Decay results were obtained 
for the so-called illuminated from interior and illuminated from exterior obstacles (see \cite{BK74}, \cite{BK76}, \cite{L87}). 
Furthermore, scattering results were recently obtained for the 3D critical nonlinear wave equation in domains exterior to such obstacles (\cite{Farah-wave}).
However, we opted to work here with almost star-shaped obstacles and use the gauge function of the ellipse
rather than the illuminating geometry (that would impose using the distance to the ellipse) mainly because the computation is much
easier with the gauge function. The dog bone like obstacle in Figure \ref{fig:os de chien} below is an almost star-shaped obstacle (and also illuminated from interior).

\begin{figure}[htbp]
\centering
 \includegraphics[width=9cm]{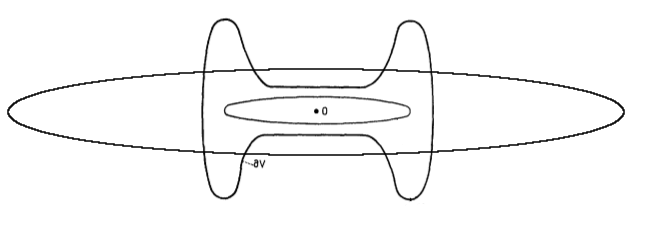}
\caption{dog bone}
\label{fig:os de chien}
\end{figure}

\subsubsection{Space-time control of the solution}
In this part, we will prove that the norm of $u$ in some $L^{a}_{t}L^{b}_{x}$ is controlled by a constant depending on the mass and the energy. This 
will be a consequence of the following proposition which is an alternative to Proposition \ref{L4L8} that is restricted to the star-shaped case:
\begin{proposition}\label{D-1/2}
Let $\Omega$ be $\mathbb{R}^{2}\setminus V$, with $V$ is an obstacle satisfying condition \eqref{geometric-condition}. 
Assume $u$ is a solution to
\begin{align*}
&i\partial_{t}u+\Delta u=\alpha|u|^{p-1}u\;\;\;\;\mbox{in}\;\Omega,\;\;p>1\\
&u|_{\mathbb{R}\times\partial\Omega}=0,
\end{align*}
with $\alpha=\{0,1\}$. Then we have 
\begin{equation}
\label{crucialR4}
\|D^{-1/2}(|v|^{2})\|_{L^{2}_{t,X}}\lesssim M^{7/4}E^{1/4}
\end{equation}
where $v(X)=v(x,y)=u(x)u(y)$ is the solution to 
\begin{align*}
&i\partial_{t} v+\Delta v=\alpha (|u|^{p-1}(x)+|u|^{p-1}(y))v\;\;\;\;\mbox{in}\;\Omega\times\Omega\\
&v|_{\partial(\Omega\times\Omega)}=0,
\end{align*}
and where we extend $v(\cdot)$ by zero for $x\not\in \Omega$ or
$y\not\in\Omega$, so that \eqref{crucialR4} makes sense for $x\in\mathbb{R}^4$.
\end{proposition}
This proposition means that (the extension to $\mathbb{R}^4$ of) $|v|^{2}\in L^{2}_{t}\dot{H}^{-1/2}_{X}$ and its norm is controlled by a constant depending on the mass and the energy 
of the solution $u$.\newline
From now on we will use the notation $C(M,E)$ to denote a constant
that depends on the conserved mass and energy of $u$. This constant
may vary from line to line. Moreover, all implicit constants are
allowed to depend on the geometry of the obstacle (in particular, they
may and will depend on $\epsilon$ appearing in
\eqref{geometric-condition}). Finally, we also have:
\begin{lemma} Let $v$ be again the extension by zero of our solution $v$ to
  the whole space $\mathbb{R}^4$. Then $|v|^{2}\in L^{\infty}_{t}H^{s}_{X}$, $\forall 0<s<1$ and its norm is controlled by $C(M,E)$.
\end{lemma}
\begin{proof}
We have $u\in H^{1}$ thus, $\forall 0<s<1$, $u\in H^{s}$ and consequently (by Sobolev embedding), 
$u\in L^{m}$ for all $m<\infty$. Now, given any $2<p<\infty$, we can easily prove that $|u|^{2}\in L^{p}(\mathbb{R}^{2})$ and $|u|^{2}\in W^{1,q}(\mathbb{R}^{2})$ 
with $1/q=1/2+1/p$. Hence, by Sobolev interpolation inequality, $|u|^{2}\in H^{1-2/p}(\mathbb{R}^{2})$. So, for any $0<s<1$, we have $|u|^{2}\in H^{s}$ and 
its norm is controlled by $C(M,E)$. Now, we have
\begin{align*}
\||v|^{2}\|_{\dot{H}^{s}}^{2}&=\int_{\mathbb{R}^{4}}(|\xi|^{2}+|\zeta|^{2})^{s}|\widehat{|v|^{2}}(\xi,\zeta)|^{2}d\xi d\zeta\\
&\leq C_{s}\int_{\mathbb{R}^{4}}(|\xi|^{2s}+|\zeta|^{2s})|\widehat{|u|^{2}}(\xi)|^{2}|\widehat{|u|^{2}}(\zeta)|^{2}d\xi d\zeta\\
&\leq 2C_{s}\||u|^{2}\|^{2}_{\dot{H}^{s}(\mathbb{R}^{2})}\|u\|^{4}_{L^{4}(\mathbb{R}^{2})}\leq C(M,E)
\end{align*}
and it is easy to see that $$\||v|^{2}\|_{L^{2}(\mathbb{R}^{4})}=\|u\|_{L^{4}(\mathbb{R}^{2})}^{4}\leq C(M,E)$$
\end{proof}
Fix $0<s<1$ to be chosen later, we have 
$$\||v|^{2}\|_ {L^{\frac{1+2s}{s}}_{t}L^{2}_{X}}\lesssim\||v|^{2}\|_{L^{2}_{t}\dot{H}^{-1/2}_{X}}^{\frac{2s}{1+2s}}\||v|^{2}\|_{ L^{\infty}_{t}H^{s}_{X}}^{1-\frac{2s}{1+2s}}\leq C(M,E)$$
Consequently, we get our desired control (which now makes sense
irrespective of $x\in \mathbb{R}^2$ or $x\in \Omega$) $$\|u\|_{L^{\frac{4(1+2s)}{s}}_{t}L^{4}_{x}}\leq C(M,E).$$
Now, we are ready to continue the proof which is practically the same as in section \ref{star-shaped case}. The solution $u$ is such that
$$u\in L^{\frac{4(1+2s)}{s}}_{I}L^{4}_{x}\cap L^{\infty}_{I}\dot{H}^{1}$$
So, 
$$u\in L^{\frac{4(1+2s)}{s}}_{I}(\dot{B}^{-1/2,\infty}_{\infty})\cap L^{\infty}_{I}(\dot{B}^{0,\infty}_{\infty})$$

Using the well known interpolation inequalities for Lebesgue and Besov spaces, we get that
\begin{equation*}
\|u\|_{ L^{q}_{I}(\dot{B}^{\gamma,\infty}_{\infty})}\leq C(M,E) 
\end{equation*}
with
$q=\frac{4(1+2s)}{s\alpha}$ and $\gamma=\frac{-\alpha}{2}$ for any $\alpha\in]0,1[$. Here, the convenient choice is $\alpha=\frac{4}{3}\frac{(1-\epsilon_{0})(1+2s)}{1+3s}$ based on the scaling of the space $L^{\frac{3}{1-\epsilon_{0}}}_{T}L^{\infty}_{x}$. 
However, to assure that $0<\alpha<1$, we need to choose $s$ such that $\frac{1-4\epsilon_{0}}{1+8\epsilon_{0}}<s<1$. 
Note that when $1/4\leq\epsilon_{0}<1$ ($p\geq5$) any $0<s<1$ will do, but for $0<\epsilon_{0}<1/4$ ($4<p<5$), 
we have a restriction on the choice of $s$.
Now, as in section \ref{star-shaped case}, we decompose any given interval $I$ of time where the solution exists: 
Given any $\eta>0$ there is a finite number of disjoint intervals $I_{1},\cdots I_{N}$ 
such that $\displaystyle\bigcup_{j=1}^{N}I_{j}=I$ with $N=N(\eta)$ and 
$$\|u\|_{L^{q}_{I_{j}}\dot{B}^{\gamma,\infty}_{\infty}}=\eta,\;\;j<N(\eta),\;\mbox{and}\;\|u\|_{L^{q}_{I_{j}}\dot{B}^{\gamma,\infty}_{\infty}}\leq\eta,\;\;j=N(\eta).$$
On the other hand, recalling that $J_{j}(t)=[t_{j},t]$, \eqref{thesecond} still holds
\begin{equation*}
\|u\|_{L^{\frac{3}{1-\epsilon}}_{J_{j}(t)}\dot{B}^{\frac{2(\epsilon_{0}-\epsilon)}{3},1}_{\infty}}\lesssim\|u(t_{j})\|_{\dot{B}^{s_{c},1}_{2}}
+\|u\|^{\frac{3}{1-\epsilon_{0}}}_{L^{\frac{3}{1-\epsilon_{0}}}_{J_{j}(t)}L^{\infty}_{x}}\|u\|_{L^{\infty}_{J_{j}(t)}\dot{B}^{s_{c},1}_{2}}
\end{equation*}
We choose $\epsilon=\frac{s\epsilon_{0}}{1+3s}<\epsilon_{0}$ and we get the following inequality
$$\|u\|_{L^{\infty}}\lesssim\|u\|^{\epsilon_{0}}_{\dot{B}^{\gamma,\infty}_{\infty}}\|u\|^{1-\epsilon_{0}}_{\dot{B}^{\frac{2}{3}(\epsilon_{0}-\epsilon),1}_{\infty}}$$
hence
\begin{equation*}
\|u\|_{L^{\frac{3}{1-\epsilon_{0}}}_{J_{j}(t)}L^{\infty}_{x}}\lesssim\|u\|^{\epsilon_{0}}_{L^{q}_{J_{j}(t)}\dot{B}^{\gamma,\infty}_{\infty}}\|u\|^{1-\epsilon_{0}}_{L^{\frac{3}{1-\epsilon}}_{J_{j}(t)}\dot{B}^{\frac{2}{3}(\epsilon_{0}-\epsilon),1}_{\infty}}
\end{equation*}
and the rest follows exactly as in section \ref{star-shaped case}.

\subsubsection{Proof of Proposition \ref{D-1/2}}
In this section we will provide the proof of Proposition \ref{D-1/2} following an approach similar to one used by Planchon and Vega in \cite{PV2D} 
to prove Proposition \ref{L4L8}. First, we will state the following remark that will be useful in our computations:
\begin{rem}\label{thebilaplacian}
 If $H$ is a function in $\mathbb{R}^{2n}$ of the form
$$H(x)=\sqrt{x_{1}^{2}+\cdots+x_{n}^{2}+\epsilon(x_{n+1}^{2}+\cdots+x_{2n}^{2})}$$ 
with $0<\epsilon<1$. Then,
$$\Delta^{2}H=\frac{A}{H^{3}}+\frac{B(x_{n+1}^{2}+\cdots+x_{2n}^{2})}{H^{5}}+\frac{C(x_{n+1}^{2}+\cdots+x_{2n}^{2})^{2}}{H^{7}}$$
with 
\begin{align*}
&A=-n(n+2)\epsilon^{2}-2n(n-3)\epsilon-n^{2}+4n-3\\
&B=2\epsilon(\epsilon-1)(3(\epsilon+1)(n+2)-15)\\
&C=-15\epsilon^{2}(\epsilon-1)^{2}<0
\end{align*}
Moreover, when $n\geq 3$ then $A,B<0$ $\forall 0<\epsilon<1$, and hence $\Delta^{2}H<0$.
\end{rem}

Now, we have the following proposition:

\begin{proposition}\label{boundarycontrol1}
Let $\Omega$ be $\mathbb{R}^{2}\setminus V$, with $V$ is an obstacle satisfying condition \eqref{geometric-condition}. 
Assume $u$ is a solution to
\begin{align*}
&i\partial_{t}u+\Delta u=\alpha|u|^{p-1}u\;\;\;\;\mbox{in}\;\Omega,\;\;p>1\\
&u|_{\mathbb{R}\times\partial\Omega}=0,
\end{align*}
with $\alpha=\{0,1\}$. Then we have the following estimate
\begin{equation}\label{localsmoothing}
\int\int_{\partial\Omega\times\Omega\times\Omega\times\Omega}\frac{|\partial_{n_{\partial V}}u(x)|^{2}}{\rho_{1}(x,y,z,w)}|u(y)|^{2}|u(z)|^{2}|u(w)|^{2}d\sigma_{x}dydzdw dt\lesssim M^{7/2}E^{1/2}
\end{equation}
where $$\rho_{1}(x,y,z,w)=\sqrt{x_{1}^{2}+y_{1}^{2}+z_{1}^{2}+w_{1}^{2}+\epsilon(x_{2}^{2}+y_{2}^{2}+z_{2}^{2}+w_{2}^{2})}$$
and $M$ and $E$ are the conserved mass and energy.
\end{proposition}
\begin{proof}
First, define $v(x,y)=u(x)u(y)$ solution to the problem
\begin{align*}
&i\partial_{t} v+\Delta v=\alpha (|u|^{p-1}(x)+|u|^{p-1}(y))v\;\;\;\;\mbox{in}\;\Omega\times\Omega\\
&v|_{\partial(\Omega\times\Omega)}=0,
\end{align*}
For star-shaped obstacles, in order to obtain local smoothing near the boundary, Planchon and Vega (\cite{PV2D}) considered 
$$\int_{\Omega\times\Omega}|v|^{2}(x,y,t)h(x,y)dxdy$$ with $h(x,y)=\sqrt{|x|^{2}+|y|^{2}}$ and computed the double derivative with respect 
to time of the 4D integral thus overcoming the problem of the wrong sign of the bilaplacian in 2D. To generalize their procedure to 
obstacles satisfying condition \eqref{geometric-condition}, we should take a weight of the form $\sqrt{x_{1}^{2}+\epsilon x_{2}^{2}+y_{1}^{2}+\epsilon y_{2}^{2}}$ 
to ensure that the boundary term has a right sign. However, this will not be enough to cover all epsilons with $ 0<\epsilon<1$ 
since the bilaplacian will not always have the right sign (see Remark \ref{thebilaplacian}). 
 This problem can be solved by increasing the dimension through applying the tensor product technique again. Remark that to ensure a 
right sign of the bilaplacian it is enough to be in 6D; but to preserve the symmetry of the computations (which is essential in Proposition \ref{D-1/2}), we 
will apply the tensor product technique again for $v$. 
Thus we define $U(x,y,z,w)=v(x,y)v(z,w)=u(x)u(y)u(z)u(w)$ solution to the 8D problem
\begin{align*}
&i\partial_{t} U+\Delta U=\alpha N(u)U\;\;\;\;\mbox{in}\;\Omega\times\Omega\times\Omega\times\Omega\\
&U|_{\partial(\Omega\times\Omega\times\Omega\times\Omega)}=0,
\end{align*}
with $$N(u)=|u|^{p-1}(x)+|u|^{p-1}(y)+|u|^{p-1}(z)+|u|^{p-1}(w).$$
Now, we consider

$$M_{\rho_{1}}(t)=\int_{\Omega\times\Omega\times\Omega\times\Omega}|U|^{2}(x,y,z,w,t)\rho_{1}(x,y,z,w)dxdydzdw$$

for 
$$\rho_{1}=\sqrt{x_{1}^{2}+y_{1}^{2}+z_{1}^{2}+w_{1}^{2}+\epsilon(x_{2}^{2}+y_{2}^{2}+z_{2}^{2}+w_{2}^{2})}$$
with $x=(x_{1},x_{2}),\;y=(y_{1},y_{2}),\;z=(z_{1},z_{2}),\;w=(w_{1},w_{2}).$

and we compute $\frac{d^{2}}{dt^{2}}M_{\rho_{1}}(t)$. This is a standard computation and similar to the one \cite{PV09} and \cite{PV2D}, up to slight 
modifications to the nonlinear term. We replicate this computation here so that the argument will be self-contained:
We have 
$$i\partial_{t}(|U|^{2})=U\Delta\overline{U}-\overline{U}\Delta U=\div(U\nabla \overline{U}-\overline{U}\nabla U)=-2i\div(Im \overline{U}\nabla U)$$
hence, by integration by parts and using the Dirichlet boundary condition we get
$$\frac{d}{dt}M_{\rho_{1}}(t)=-2Im\int \rho_{1}\div(\overline{U}\nabla U)=2 Im\int\overline{U}\nabla U\cdot\nabla \rho_{1}$$
Now,
\begin{align*}
\frac{d^{2}}{dt^{2}}M_{\rho_{1}}(t)&=2 Im\int(\partial_{t}\overline{U}\nabla U+\overline{U}\nabla\partial_{t}U)\cdot\nabla \rho_{1}=-2Im\int\partial_{t}U(2\nabla\overline{U}\cdot\nabla \rho_{1}+\overline{U}\Delta \rho_{1})\\
&=-2Re\int(\Delta U-\alpha N(u)U)(2\nabla\overline{U}\cdot\nabla \rho_{1}+\overline{U}\Delta \rho_{1})\\
&=-4Re\int\Delta U\nabla\overline{U}\cdot\nabla \rho_{1}+2\int|\nabla U|^{2}\Delta \rho_{1}+2Re\int\overline{U}\nabla U\cdot\nabla(\Delta \rho_{1})\\
&\;\;+2\alpha\int N(u)\nabla(|U|^{2})\nabla \rho_{1}+2\alpha\int N(u)|U|^{2}\Delta \rho_{1}\\
&=-4Re\int\Delta U\nabla\overline{U}\cdot\nabla \rho_{1}+2\int|\nabla U|^{2}\Delta \rho_{1}-\int|U|^{2}\Delta^{2}\rho_{1}-2\alpha\int|U|^{2}\nabla N\cdot\nabla \rho_{1}.
\end{align*}
Integrating by parts again,
$$\int\Delta U\nabla\overline{U}\cdot\nabla \rho_{1}=-\int|\partial_{n}U|^{2}\partial_{n}\rho_{1}-\int\nabla U\cdot\nabla(\nabla\overline{U}\cdot\nabla \rho_{1})$$
where $n$ is the normal pointing into the domain.
Thus
\begin{align*}
2Re\int\Delta U\nabla\overline{U}\cdot\nabla \rho_{1}&=-2\int|\partial_{n}U|^{2}\partial_{n}\rho_{1}-\int\nabla(|\nabla U|^{2})\cdot\nabla \rho_{1}-2\int Hess \rho_{1}(\nabla U,\nabla\overline{U})\\
&=-2\int|\partial_{n}U|^{2}\partial_{n}\rho_{1}+\int|\nabla U|^{2}\Delta \rho_{1}-2\int Hess \rho_{1}(\nabla U,\nabla\overline{U})
\end{align*}
Moreover, by integrating by parts we have
\begin{align*}
&-2\alpha\int|U|^{2}\nabla N\cdot\nabla \rho_{1}\\
&=\frac{2(p-1)}{p+1}\alpha\int|U|^{2}(|u|^{p-1}(x)\Delta_{x}\rho_{1}+|u|^{p-1}(y)\Delta_{y}\rho_{1}+|u|^{p-1}(z)\Delta_{z}\rho_{1}+|u|^{p-1}(w)\Delta_{w}\rho_{1})
\end{align*}
and we finally obtain
\begin{align}\label{Mh}
&\frac{d^{2}}{dt^{2}}M_{\rho_{1}}(t)=-\int|U|^{2}\Delta^{2}\rho_{1}+2\int|\partial_{n}U|^{2}\partial_{n}\rho_{1}+4\int Hess \rho_{1}(\nabla U,\nabla\overline{U})\\
&+\frac{2(p-1)}{p+1}\alpha\int|U|^{2}(|u|^{p-1}(x)\Delta_{x}\rho_{1}+|u|^{p-1}(y)\Delta_{y}\rho_{1}+|u|^{p-1}(z)\Delta_{z}\rho_{1}+|u|^{p-1}(w)\Delta_{w}\rho_{1})\notag
\end{align}
From our choice of the convex function $\rho_{1}$ we have 
that the terms with the Hessian as well as those with the Laplacian are positive.\newline
We also have from Remark \ref{thebilaplacian} that 8D bilaplacian ($n=4$) $\Delta^{2}\rho_{1}$ is negative $\forall 0<\epsilon<1$.
Now, we deal with boundary term. First, we look at the term $\partial_{n}\rho_{1}$ with $n$ the 
normal pointing into $\Omega\times\Omega\times\Omega\times\Omega$, we have
\begin{align*}
&n=(n_{x},0,0,0)\;\;\;\;\mbox{if}\;x\in\partial\Omega,\;y,z,w\in\Omega\\
&n=(0,n_{y},0,0)\;\;\;\;\mbox{if}\;y\in\partial\Omega,\;x,z,w\in\Omega\\
&n=(0,0,n_{z},0)\;\;\;\;\mbox{if}\;z\in\partial\Omega,\;x,y,w\in\Omega\\
&n=(0,0,0,n_{w})\;\;\;\;\mbox{if}\;w\in\partial\Omega,\;x,y,z\in\Omega
\end{align*}
Hence, if $x\in\partial\Omega$
$$\partial_{n}\rho_{1}=\frac{(x_{1},\epsilon x_{2})\cdot n_{x}}{\rho_{1}}$$
which is strictly positive by the geometric condition we imposed \eqref{geometric-condition}. Moreover,
$$\partial_{n}\rho_{1}\geq\frac{C}{\rho_{1}}$$
and we also have $$|\partial_{n}U|^{2}=|\partial_{n_{x}}u(x)|^{2}|u(y)|^{2}|u(z)|^{2}|u(w)|^{2},$$ 
and we deal similarly when $y,z,\;\mbox{or}\;w\in\partial\Omega$.
Hence, \eqref{Mh} yields
\begin{align*}
\int\int_{\partial\Omega\times\Omega\times\Omega\times\Omega}\frac{|\partial_{n_{x}}u(x)|^{2}}{\rho_{1}}|u(y)|^{2}|u(z)|^{2}|u(w)|^{2}d\sigma dt
\lesssim M^{7/2}E^{1/2}
\end{align*}
which ends the proof of Proposition \ref{boundarycontrol1}.
\end{proof}
\vspace{3mm}
Due to the fact that we are doing the tensor product technique more than once, and we are dealing now with four 2D variables, we will need extra estimates on the boundary. 
We have the following proposition:
\begin{proposition}\label{boundarycontrol2}
  Under the conditions of Proposition \ref{boundarycontrol1}, we have the following estimate
\begin{equation}\label{1-2}
\int\int_{\partial\Omega\times\Omega\times\Omega\times\Omega}\frac{|\partial_{n_{x}}u(x)|^{2}|u(y)|^{2}|u(z)|^{2}|u(w)|^{2}}{\sqrt{|x|^{2}+|z|^{2}+|y\pm w|^{2}}}d\sigma_{x} dydzdwdt\lesssim M^{7/2}E^{1/2}
\end{equation}
\end{proposition}
\begin{rem}
 Remark that Proposition \ref{boundarycontrol2} is obviously improving over Proposition \ref{boundarycontrol1}, as the new weight has
less decay  in some
directions (actually, no decay in direction $y-w$ or $y+w$ for
example!), whereas $\rho_{1}$ is uniformly decaying in all directions.
\end{rem}

\begin{proof}
To prove the estimates \eqref{1-2}, we do the same standard procedure as in Proposition \ref{boundarycontrol1} with the weight 
$\rho_{2}$ defined as
\begin{align*}
\rho_{2}&=\sqrt{x_{1}^{2}+\epsilon x_{2}^{2}+z_{1}^{2}+\epsilon z_{2}^{2}+\left(\frac{y_{1}-w_{1}}{\sqrt{2}}\right)^{2}+\epsilon\left(\frac{y_{2}-w_{2}}{\sqrt{2}}\right)^{2}}\\
&\;+\sqrt{x_{1}^{2}+\epsilon x_{2}^{2}+z_{1}^{2}+\epsilon z_{2}^{2}+\left(\frac{y_{1}+w_{1}}{\sqrt{2}}\right)^{2}+\epsilon\left(\frac{y_{2}+w_{2}}{\sqrt{2}}\right)^{2}}\\
&:=\rho_{2}^{-}+\rho_{2}^{+}
\end{align*}
Again, we consider
$$M_{\rho_{2}}(t)=\int_{\Omega\times\Omega\times\Omega\times\Omega}|U|^{2}(x,y,z,w,t)\rho_{2}(x,y,z,w)dxdydzdw$$
and we compute $\frac{d^{2}}{dt^{2}}M_{\rho_{2}}(t)$
to get 
\begin{align}\label{Mrho2}
&\frac{d^{2}}{dt^{2}}M_{\rho_{2}}(t)=-\int|U|^{2}\Delta^{2}\rho_{2}+2\int|\partial_{n}U|^{2}\partial_{n}\rho_{2}+4\int Hess \rho_{2}(\nabla U,\nabla\overline{U})\\
&+\frac{2(p-1)}{p+1}\alpha\int|U|^{2}(|u|^{p-1}(x)\Delta_{x}\rho_{2}+|u|^{p-1}(y)\Delta_{y}\rho_{2}+|u|^{p-1}(z)\Delta_{z}\rho_{2}+|u|^{p-1}(w)\Delta_{w}\rho_{2})\notag
\end{align}
Note that $\rho_{2}$ is convex thus the Hessian is positive, and the terms with the Laplacian are positive as well. 
As for the term of the bilaplacian, note that the functions $\rho_{2}^{-}$ and $\rho_{2}^{+}$ of $(x,y,z,w)$ can be also viewed as
functions of $$(x,z,\frac{y-w}{\sqrt{2}},\frac{y+w}{\sqrt{2}}):=(\xi_{1},\xi_{2},\xi_{3},\xi_{4})$$ with
$\nabla_{\xi_{3}}\rho_{2}^{+}=0$ and $\nabla_{\xi_{4}}\rho_{2}^{-}=0$. Since the bilaplacian in invariant under rotation, we have
$$\Delta^{2}_{x,y,z,w}\rho_{2}^{-}=\Delta^{2}_{\xi_{1},\xi_{2},\xi_{3},\xi_{4}}\rho_{2}^{-}
=\Delta^{2}_{\xi_{1},\xi_{2},\xi_{3}}\left(\sqrt{\xi_{11}^{2}+\xi_{21}^{2}+\xi_{31}^{2}+\epsilon(\xi_{12}^{2}+\xi_{22}^{2}+\xi_{32}^{2})}\right)$$
and by Remark \ref{thebilaplacian}, this 6D bilaplacian ($n=3$) is negative.
Similarly, $\Delta^{2}\rho_{2}^{+}<0$, hence we have $\Delta^{2}\rho_{2}<0$.\newline
Now, we deal the boundary term in \eqref{Mrho2}.
First, we want to control the terms we get on the boundary when $(y,w)\in\partial(\Omega\times\Omega)$. If $y\in\partial\Omega$ then 
$$\nabla_{y}\rho_{2}=\frac{1}{2\rho_{2}^{-}}(y_{1}-w_{1},\epsilon(y_{2}-w_{2}))+\frac{1}{2\rho_{2}^{+}}(y_{1}+w_{1},\epsilon(y_{2}+w_{2}))$$
Introduce
$$\gamma=\sqrt{x_{1}^{2}+\epsilon x_{2}^{2}+z_{1}^{2}+\epsilon z_{2}^{2}+\frac{1}{2}(y_{1}^{2}+\epsilon y_{2}^{2}+w_{1}^{2}+\epsilon w_{2}^{2})}$$
Thus, $$\rho_{2}^{-}=\gamma\sqrt{1-\frac{y_{1}w_{1}+\epsilon y_{2}w_{2}}{\gamma^{2}}}$$
and $$\rho_{2}^{+}=\gamma\sqrt{1+\frac{y_{1}w_{1}+\epsilon y_{2}w_{2}}{\gamma^{2}}}$$
now, we write
$$\frac{1}{\rho_{2}^{\pm}}=\frac{1}{\gamma}+\frac{1}{\gamma}\left(\frac{1}{\sqrt{1\pm\frac{y_{1}w_{1}+\epsilon y_{2}w_{2}}{\gamma^{2}}}}-1\right)$$
and substitute in $\nabla_{y}\rho_{2}$ to get
\begin{align*}
\nabla_{y}\rho_{2}&=\frac{(y_{1},\epsilon y_{2})}{\gamma}
+\frac{(y_{1},\epsilon y_{2})}{2}\left[\frac{1}{\gamma}\left(\frac{1}{\sqrt{1-\frac{y_{1}w_{1}+\epsilon y_{2}w_{2}}{\gamma^{2}}}}-1\right)
+\frac{1}{\gamma}\left(\frac{1}{\sqrt{1+\frac{y_{1}w_{1}+\epsilon y_{2}w_{2}}{\gamma^{2}}}}-1\right)\right]\\
&\;\;\;\;+\frac{(w_{1},\epsilon w_{2})}{2}\left[\frac{1}{\gamma}\left(\frac{1}{\sqrt{1+\frac{y_{1}w_{1}+\epsilon y_{2}w_{2}}{\gamma^{2}}}}-1\right)
-\frac{1}{\gamma}\left(\frac{1}{\sqrt{1-\frac{y_{1}w_{1}+\epsilon y_{2}w_{2}}{\gamma^{2}}}}-1\right)\right]
\end{align*}
Using the fact the $y$ is bounded and $\gamma$ is large enough, there exists a positive constant $c$ such that
$$\frac{|y_{1}w_{1}+\epsilon y_{2}w_{2}|}{\gamma^{2}}\leq\frac{c}{\gamma}<1$$
and thus 
$$\left|\frac{1}{\gamma}\left(\frac{1}{\sqrt{1\pm\frac{y_{1}w_{1}+\epsilon y_{2}w_{2}}{\gamma^{2}}}}-1\right)\right|\leq\left|\frac{1}{\gamma}\left(\frac{1}{\sqrt{1-\frac{c}{\gamma}}}-1\right)\right|\lesssim\frac{1}{\gamma^{2}}$$
this implies that
$$|\nabla_{y}\rho_{2}|\lesssim\frac{1}{\gamma}\leq\frac{\sqrt{2}}{\rho_{1}}$$
so, the boundary term obtained when $y\in\partial\Omega$ is controlled by Proposition \ref{boundarycontrol1}:
\begin{align*}
&\int\int_{\Omega\times\partial\Omega\times\Omega\times\Omega}|u(x)|^{2}|\partial_{n_{y}}u(y)|^{2}|u(z)|^{2}|u(w)|^{2}|\partial_{n}\rho_{2}|dx d\sigma_{y}dzdwdt\\
&\;\lesssim\int\int_{\Omega\times\partial\Omega\times\Omega\times\Omega}\frac{|u(x)|^{2}|\partial_{n_{y}}u(y)|^{2}|u(z)|^{2}|u(w)|^{2}}{\rho_{1}}dx d\sigma_{y}dzdwdt\lesssim M^{7/2}E^{1/2}
\end{align*}
similarly for the boundary term generated when $w\in\partial\Omega$.\newline
Now, when $x\in\partial\Omega$, then 
$$\partial_{n}\rho_{2}=\nabla_{x}\rho_{2}\cdot n_{x}=\left(\frac{1}{\rho_{2}^{-}}+\frac{1}{\rho_{2}^{+}}\right)(x_{1},\epsilon x_{2})\cdot n_{x}$$
Again by the geometry of the obstacle, we have $(x_{1},\epsilon x_{2})\cdot n_{x}>0$ and thus 
$$\partial_{n}\rho_{2}\gtrsim\frac{1}{\rho_{2}^{-}}+\frac{1}{\rho_{2}^{+}}\gtrsim_\varepsilon\frac{1}{\sqrt{|x|^{2}+|z|^{2}+|y-w|^{2}}}+\frac{1}{\sqrt{|x|^{2}+|z|^{2}+|y+w|^{2}}}$$
and $$|\partial_{n}U|^{2}=|\partial_{n_{x}}u(x)|^{2}|u(y)|^{2}|u(z)|^{2}|u(w)|^{2}.$$
and we deal similarly when $z\in\partial\Omega$.
So, finally \eqref{Mrho2} yields 
\begin{align*}
&\int\int_{\partial\Omega\times\Omega\times\Omega\times\Omega}\frac{|\partial_{n_{x}}u(x)|^{2}|u(y)|^{2}|u(z)|^{2}|u(w)|^{2}}{\sqrt{|x|^{2}+|z|^{2}+|y-w|^{2}}}d\sigma_{x}dydzdwdt\\
&+\int\int_{\partial\Omega\times\Omega\times\Omega\times\Omega}\frac{|\partial_{n_{x}}u(x)|^{2}|u(y)|^{2}|u(z)|^{2}|u(w)|^{2}}{\sqrt{|x|^{2}+|z|^{2}+|y+w|^{2}}} d\sigma_{x}dydzdwdt\lesssim M^{7/2}E^{1/2}.
\end{align*}
\end{proof}
Now, we are ready to prove Proposition \ref{D-1/2}.
Again, we proceed in a similar argument to that in previous propositions, we consider
$$M_{\rho}(t)=\int_{(\Omega\times\Omega)\times(\Omega\times\Omega)}|U|^{2}(X,Y,t)\rho(X,Y)dXdY$$
where $U(X,Y)=v(X)v(Y)=u(x)u(y)u(z)u(w)$, with $X=(x,y),\;Y=(z,w)\in\Omega\times\Omega$, is the solution to the problem
\begin{align*}
&i\partial_{t} U+\Delta U=\alpha N(u)U\;\;\;\;\mbox{in}\;\Omega\times\Omega\times\Omega\times\Omega\\
&U|_{\partial(\Omega\times\Omega\times\Omega\times\Omega)}=0,
\end{align*}
with $$N(u)=|u|^{p-1}(x)+|u|^{p-1}(y)+|u|^{p-1}(z)+|u|^{p-1}(w)$$
for 
\begin{align*}
\rho(X,Y)=&|X-Y|+|X'+Y|+|X'-Y|+|X+Y|\\
=&\sqrt{|x-z|^{2}+|y-w|^{2}}+\sqrt{|x+z|^{2}+|y-w|^{2}}\\
&+\sqrt{|x-z|^{2}+|y+w|^{2}}+\sqrt{|x+z|^{2}+|y+w|^{2}}
\end{align*}
where $X'=(x,-y)$.
Doing the same standard computation, we get
\begin{align}\label{Mrho}
&\frac{d^{2}}{dt^{2}}M_{\rho}(t)=-\int|U|^{2}\Delta^{2}\rho+2\int|\partial_{n}U|^{2}\partial_{n}\rho+4\int Hess\rho(\nabla U,\nabla\overline{U})\\
&+\frac{2(p-1)}{p+1}\alpha\int|U|^{2}(|u|^{p-1}(x)\Delta_{x}\rho+|u|^{p-1}(y)\Delta_{y}\rho+|u|^{p-1}(z)\Delta_{z}\rho+|u|^{p-1}(w)\Delta_{w}\rho)\notag
\end{align}
The weight $\rho$ is convex and thus the Hessian term and the Laplacian terms are positive. 
Moreover, we have
$$-\Delta^{2}\rho=12\left(\frac{1}{|X-Y|^{3}}+\frac{1}{|X'+Y|^{3}}+\frac{1}{|X'-Y|^{3}}+\frac{1}{|X+Y|^{3}}\right)$$
Now, we control the boundary term. If $x\in\partial\Omega$ and $y,z,w\in\Omega$ then 
$\partial_{n}\rho=\nabla_{x}\rho\cdot n_{x}$ and we have
\begin{align*}
\nabla_{x}\rho&=\frac{x-z}{|X-Y|}+\frac{x+z}{|X'+Y|}+\frac{x-z}{|X'-Y|}+\frac{x+z}{|X+Y|}
\end{align*}
Setting $\lambda_{-}^{2}=|x|^{2}+|z|^{2}+|y-w|^{2}$ and $\lambda_{+}^{2}=|x|^{2}+|z|^{2}+|y+w|^{2}$, we have:
\begin{align*}
|X-Y|^{2}&=\lambda_{-}^{2}\left(1-\frac{2x\cdot z}{\lambda_{-}^{2}}\right),\;\;|X'+Y|^{2}=\lambda_{-}^{2}\left(1+\frac{2x\cdot z}{\lambda_{-}^{2}}\right)\\
|X'-Y|^{2}&=\lambda_{+}^{2}\left(1-\frac{2x\cdot z}{\lambda_{+}^{2}}\right),\;\;|X+Y|^{2}=\lambda_{+}^{2}\left(1+\frac{2x\cdot z}{\lambda_{+}^{2}}\right)
\end{align*}
Reasoning as in Proposition \ref{boundarycontrol2}, we write 
$$\frac{1}{|X\pm Y|}=\frac{1}{\lambda_{\pm}}+\frac{1}{\lambda_{\pm}}\left(\frac{1}{\sqrt{1\pm\frac{2x\cdot z}{\lambda_{\pm}^{2}}}}-1\right)$$
and 
$$\frac{1}{|X'\pm Y|}=\frac{1}{\lambda_{\mp}}+\frac{1}{\lambda_{\mp}}\left(\frac{1}{\sqrt{1\pm\frac{2x\cdot z}{\lambda_{\mp}^{2}}}}-1\right)$$
and we substitute in $\nabla_{x}\rho$ which yields some convenient cancellations in the $z$ terms. Then, using the fact that $|x|$ is under control and $\lambda_{\pm}$ are large enough, there 
exists a positive constant $c'$ such that
$$\frac{|2x\cdot z|}{\lambda_{\pm}^{2}}\leq\frac{c'}{\lambda_{\pm}}<1$$ and thus 
$$\left|\frac{1}{\lambda_{-}}\left(\frac{1}{\sqrt{1\pm\frac{2x\cdot z}{\lambda_{-}^{2}}}}-1\right)\right|\leq\left|\frac{1}{\lambda_{-}}\left(\frac{1}{\sqrt{1-\frac{c'}{\lambda_{-}}}}-1\right)\right|\lesssim\frac{1}{\lambda_{-}^{2}}$$
and similarly for $\lambda_{+}$. This yields that 
$$|\nabla_{x}\rho|\lesssim\frac{1}{\lambda_{-}}+\frac{1}{\lambda_{+}},$$ and thus 
the boundary term generated when $x\in\partial\Omega$ is controlled by \eqref{1-2} of Proposition \ref{boundarycontrol2}, 
and similarly when $z\in\partial\Omega$.\newline 
Now, when $y\in\partial\Omega$ or $w\in\partial\Omega$, we do a similar procedure but with 
$$\widetilde{\lambda}_{\pm}^{2}=|y|^{2}+|w|^{2}+|x\pm z|^{2},$$ and we get the same control on the boundary terms by 
Proposition \ref{boundarycontrol2}.
So, finally, \eqref{Mrho} yields
$$\int\int_{(\Omega\times\Omega)\times(\Omega\times\Omega)}\frac{|v(X)|^{2}|v(Y)|^{2}}{|X-Y|^{3}}dXdYdt\lesssim M^{7/2}E^{1/2}$$
which actually holds on $\mathbb{R}^4\times \mathbb{R}^4$ provided we
extend $v$ by zero
inside the obstacle. Then, by Plancherel's theorem we get 
$$\|D^{-1/2}(|v|^{2})\|^{2}_{L^{2}_{t,X}}$$
which ends the proof of Proposition \ref{D-1/2}.

\end{document}